\documentclass[12pt]{amsart}
\usepackage{amsrefs}
\usepackage[english]{babel}
\usepackage[utf8]{inputenc}
\usepackage{amsmath,relsize,hyperref}
\usepackage{graphicx}
\usepackage{amssymb}
\usepackage{amsthm}
\usepackage{float}

\numberwithin{equation}{section}

\usepackage{tikz-cd}
\tikzset{
  symbol/.style={
    draw=none,
    every to/.append style={
      edge node={node [sloped, allow upside down, auto=false]{$#1$}}}
  }
}

\usepackage{microtype}
\usepackage{mathrsfs}
\usepackage[colorinlistoftodos]{todonotes}
\usepackage{enumitem}
\usepackage{yfonts}
\usepackage{ dsfont }

\usetikzlibrary{matrix}
\usetikzlibrary{calc,intersections}
\usepackage{adjustbox}

\newtheorem{thm}{Theorem}[section]

\newtheorem{prop}[thm]{Proposition}

\newtheorem{remark}[thm]{Remark}

\newtheorem{cor}[thm]{Corollary}

\renewcommand{\L}{{\mathcal L}}

\newcommand{\B}{{\mathbb B}}
\renewcommand{\H}{{\mathcal H}}
\newcommand{\M}{{\mathcal M}}
\newcommand{\ie}{\emph{i.e.} }

\newcommand{\resp}{\emph{resp.} }
\newcommand{\Hbar}{\overline{{\mathcal H}}}
\newcommand{\Mbar}{\overline{{\mathcal M}}}
\newcommand{\HH}{{\mathbb H}}

\newcommand{\Z}{\mathbb{Z}}

\newcommand{\Cc}{\mathcal{C}}
\newcommand{\Aff}{\mathbb{A}}
\newcommand{\Pro}{\mathbb{P}}
\newcommand{\D}{\mathcal{D}}
\newcommand{\G}{\mathbb{G}}

\DeclareMathOperator{\im}{Im}

\DeclareMathOperator{\Spec}{Spec}
\DeclareMathOperator{\Pic}{Pic}

\DeclareMathOperator{\Char}{char}
\DeclareMathOperator{\Sym}{Sym}

\DeclareMathOperator{\GL}{GL}

\DeclareMathOperator{\PGL}{PGL}
\newcommand{\coloneq}{:=}

\usepackage[margin=1in]{geometry}

\begin{document}

\title[Stacks of hyperelliptic Weierstrass points]{The integral Chow rings of the stacks of hyperelliptic Weierstrass points}
\author{Dan Edidin}
\address{Department of Mathematics, University of Missouri, Columbia, MO 65211}
\email{edidind@missouri.edu}

\author{Zhengning Hu}
\address{Department of Mathematics, University of Arizona, Tucson, AZ 85721}
\email{zhengninghu@arizona.edu}

\begin{abstract}
  We compute the integral Chow rings of the stacks $\H_{g,n}^w$ parametrizing hyperelliptic curves with  $n$ marked hyperelliptic Weierstrass points. We prove that the integral
  Chow ring of each of these stacks is generated as an algebra
  by any of the $\psi$-classes and that all relations live in degree one.
  \end{abstract}

\maketitle

\section{Introduction}

The purpose of this paper is to compute the  integral Chow rings
of the stacks $\H_{g,n}^w$  parametrizing smooth hyperelliptic curves
with $n$ marked Weierstrass points. Aside from their intrinsic interest,
products of these stacks naturally arise as boundary strata of proper
moduli stacks of curves such as $\Hbar_g$, parametrizing stable hyperelliptic curves. Thus, any program to compute the integral Chow ring of $\Hbar_g$ via the stratification method of \cite{Lar:21} will require our results as input.

We now describe the moduli stacks considered in this paper.
Let $\H_g$ denote the moduli stack of smooth hyperelliptic curves, and let
$\Cc_g \to \H_g$ be the universal hyperelliptic curve. A result of Kleiman and
L\o nsted \cite{KlLo:79} implies that there is a smooth Cartier
divisor $\H_{g,w} \subset \Cc_g$ parametrizing the hyperelliptic Weierstrass points in the fibers of $\Cc_g \to \H_g$. Moreover, if $\Char k \neq 2$, then $\H_{g,w}$ is finite and \'etale of degree $2g+2$ over $\H_g$. The stack $\H_{g,w}$ which we call the {\em stack of hyperelliptic Weierstrass points} parametrizes families of smooth hyperelliptic curves with a Weierstrass section.

Let $\Cc_{g,w} \to \H_{g,w}$ be the family of hyperelliptic curves obtained by base change along the finite \'etale morphism $\H_{g,w} \to \H_g$. The pullback
of the Weierstrass divisor $\H_{g,w}$ decomposes as the disjoint union of
the universal Weierstrass section and a divisor
$\H_{g,2}^w$ which is finite and \'etale over $\H^w_{g,1}:=\H_{g,w}$ of degree $2g+1$. Iterating this construction, we obtain a tower of finite \'etale covers
\begin{equation*}
  \H_{g,2g+2}^w \to \H_{g,2g+1}^w \to \ldots \to \H_{g,w} \to \H_g
\end{equation*}
where $\H_{g,n}^w$ parametrizes hyperelliptic curves with $n$-Weierstrass sections and the morphism $\H_{g,n+1}^w \to \H_{g,n}^w$ is finite and \'etale of degree $2g+2 -n$. Each of the stacks $\H_{g,n}^w$ is a $\mu_2$-gerbe over the stack $[\M_{0,2g+2}/S_{2g+2-n}]$ which parametrizes families of rational curves with $n$-marked points and a disjoint divisor of degree $2g+2-n$ which is \'etale over the base.

Putting this together we obtain the following cartesian diagram \eqref{diag.main} where the horizontal arrows are $\mu_2$-gerbes and the vertical arrows are finite, \'etale and representable.

\begin{equation} 
\begin{tikzcd}[row sep = small] \label{diag.main}
\H_{g,2g + 2}^w \ar[r] \ar[d] & \M_{0,2g + 2} \ar[d] \\
\vdots \ar[d] & \vdots \ar[d] \\
\H_{g,2}^w \ar[r] \ar[d] & \left[\M_{0,2g + 2}/S_{2g}\right] \ar[d]\\
\H_{g,1}^w\coloneq \H_{g,w} \ar[r] \ar[d] & \left[\M_{0,2g + 2}/S_{2g + 1}\right] \ar[d] \\
\H_g \ar[r] & \left[\M_{0,2g + 2}/S_{2g + 2}\right]
  \end{tikzcd}
  \end{equation} 

The integral Chow ring of $\H_g$ was computed in the papers \cite{EdFu:07} ($g$-even), \cites{FuVi:11, diL:18} ($g$-odd). In this paper we compute the integral Chow rings of the other stacks in diagram. Note that the rational Chow rings of the stacks $\H_{g,n}^w$ are trivial
because diagram~\eqref{diag.main} implies that they are gerbes over
quotients of $\M_{0,2g+2}$ which has trivial Chow groups. Thus, the problem of computing the integral Chow rings is all the more interesting.

{\bf Acknowledgments.} The results of this article were obtained while
the authors participated in the Semester program on Moduli and
Algebraic Cycles at the Mittag--Leffler Institute in Fall 2021. The
authors are grateful to Mittag--Leffler for financial support and to
the participants for providing an excellent working environment as
well as an opportunity to present a preliminary version of this work.
They are particularly grateful to Dan Petersen for suggesting the
proof of Proposition \ref{prop.danpetersen} and to Michele Pernice for
a number of helpful discussions. The first author was also supported
by Simons collaboration grant 708560 while preparing this work.

\subsection{Statement of results}\hspace*{\fill}

{\noindent\bf Conventions and notation.} Throughout this paper, we fix a
natural number $g \geq 2$ and work over an algebraically closed field
$k$ whose characteristic does not 
divide $2g + 2$. We will use the notation $\Aff(N)$ to denote the
affine space of binary forms of degree $N$. As such $\Aff(N) \simeq
\Aff^{N+1}$. There is a natural action of $\GL_2$ on $\Aff(N)$ given by
$A \cdot f(x) =f(A^{-1}x)$. The kernel of this action is the
diagonal subgroup $\mu_N$.

As noted above, the integral Chow ring of $\H_g$ was computed by Edidin--Fulghesu, Fulghesu--Viviani and Di Lorenzo. In order to contrast it with
our work we restate their results.

\begin{thm} \label{thm.Hg}
  The integral Chow ring of the stack of hyperelliptic curves is as follows.
\begin{itemize}
\item (Edidin--Fulghesu \cite{EdFu:07}) If $g$ is even, then
\[A^*(\H_g) = \Z[c_1,c_2]/(2(2g + 1)c_1,g(g - 1)c_1^2 - 4g(g + 1)c_2).\]

\item (Fulghesu--Viviani \cite{FuVi:11}, Di Lorenzo \cite{diL:18}) If $g$ is odd, then
\[A^*(\H_g) = \Z[\tau,c_2,c_3]/(4(2g + 1)\tau,8\tau^2 - 2g(g + 1)c_2,2c_3).\]
\end{itemize}
\end{thm}

\begin{thm}  \label{thm.m2g+2}
The integral Chow ring of $[\M_{0,2g + 2}/S_{2g + 2}]$ is
\[A^*([\M_{0,2g + 2}/S_{2g + 2}]) = \dfrac{\Z[\tau,c_2,c_3]}{(2\tau^2 - 2g(g + 1)c_2,2(2g + 1)\tau,2c_3,p(\tau,c_2,c_3))}\]
where
\begin{itemize}
\item if $g$ is odd, then
\[p(\tau,c_2,c_3) = (g + 1)^2\tau^{g}(\tau^2 + c_2)^{\frac{g + 1}{2}}c_2 + \tau^{\frac{g + 3}{2}}(\tau^3 + \tau c_2 - c_3)^{\frac{g + 1}{2}};\] 
\item if $g$ is even, then
\[p(\tau,c_2,c_3) = g(g+2)\tau^{g + 1}(\tau^2 + c_2)^{\frac{g}{2}}c_2 + \tau^{\frac{g}{2}}(\tau^3 + \tau c_2 - c_3)^{\frac{g + 2}{2}}.\]
\end{itemize}
\end{thm}

Once we mark at least one Weierstrass point, the presentation of the Chow rings  no longer depends on the parity of the genus $g$.
\begin{thm}\label{thm.Hgw}
The integral Chow rings of $\H_{g,w}$ and $[\M_{0,2g + 2}/S_{2g + 1}]$ are 
\[A^*(\H_{g,w}) = \dfrac{\Z[\psi]}{(4g(2g + 1)\psi)},\quad A^*([\M_{0,2g + 2}/S_{2g + 1}]) = \dfrac{\Z[l]}{(2g(2g + 1)l)},\]
respectively. Here $\psi$ is the restriction of the $\psi$-class on $\H_{g,1}$ to $\H_{g,w}$.
\end{thm}
\begin{thm} \label{thm.Hg2w}
The integral Chow rings of $\H_{g,2}^w$ and $[\M_{0,2g + 2}/S_{2g}]$ are
\[A^*(\H_{g,2}^w) = \dfrac{\Z[\psi]}{(4g\psi)},\quad A^*([\M_{0,2g + 2}/S_{2g}]) = \dfrac{\Z[l]}{(2gl)},\]
respectively. Here $\psi$ is the restriction of either of the two $\psi$-classes,
$\psi_1, \psi_2$ on $\H_{g,2}$ to $\H_{g,2}^w$.
\end{thm}

\begin{thm}\label{thm.Hg3w}
For $3 \leq n \leq 2g + 2$, the integral Chow rings of $\H_{g,n}^w$ and $[\M_{0,2g + 2}/S_{2g + 2 - n}]$ are
\[A^*(\H_{g,n}^w) = \Z[\psi]/(2\psi),\quad A^*([\M_{0,2g + 2}/S_{2g + 2 - n}]) = \Z.\] Here $\psi$ is the restriction of any of the $\psi$-classes on $\H_{g,n}$ to $\H_{g,n}^w$.
\end{thm}

\section{Proof of Theorem \ref{thm.m2g+2}}

\subsection{Description of $[\M_{0,2g + 2}/S_{2g + 2}]$ as a quotient
  by $\G_m \times \PGL_2$}
The quotient stack $[\M_{0,2g + 2}/S_{2g + 2}]$ can be
identified with the stack $\D_{2g+2}$ parametrizing pairs $(P \to S,
D_{2g+2} \subset P)$ where $P \to S$ is a twisted $\Pro^1$ over $S$ and $D_{2g+2}$ is an
effective Cartier divisor which is finite and \'etale over $S$ of degree $2g + 2$. By \cite[Proposition 3.4]{GoVi:06}, the algebraic stack $\D_{2g+2}$ can be
identified with the quotient stack
$[\Aff_{sm}(2g+2)/(\GL_2/\mu_{2g+2})]$, where $\Aff_{sm}(2g+2)$ is
the affine space of binary forms of degree $2g+2$ with distinct roots, and the
action on $\Aff_{sm}(2g+2)$ is given by $A \cdot f(x) =
f(A^{-1}x)$.

\begin{remark}
By \cite[Theorem 4.1]{ArVi:03}, the stack $\H_g$ of smooth hyperelliptic curves of genus $g$ is isomorphic to the quotient stack
$[\Aff_{sm}(2g+2)/(\GL_2/\mu_{g+1})]$ with action defined by $A \cdot f(x) = f(A^{-1}x)$. Consider the 
short exact sequence
$$ 1\to \mu_2 \to \mu_{2g+2} \to \mu_{g+1} \to 1$$ where the second arrow is
given by 
$\alpha \mapsto \alpha^2$. There is an induced map of quotient groups
$\GL_2/\mu_{g+1} \to \GL_2/\mu_{2g+2}$ which is a $\mu_2$-torsor. This in turn induces a map of quotient stacks
$$\H_g = [\Aff_{sm}(2g+2)/(\GL_2/\mu_{g+1})] \to \D_{2g + 2} = [\Aff_{sm}(2g+2)/
  (\GL_2/\mu_{2g+2})]$$ which is a $\mu_2$-gerbe.
\end{remark}

Since $2g+2$ is always even, \cite[Proposition 4.4]{ArVi:03}
implies that the homomorphism
$\GL_2/\mu_{2g+2} \to \G_m \times \PGL_2$, given by $[A] \mapsto (\det (A)^{g+1}, [A])$ is an isomorphism. Under this identification of groups, the action of
$\G_m \times \PGL_2$ on the affine space of binary forms $\Aff(2g+2)$ is given by
$(\lambda, [A])\cdot f(x) = \lambda^{-1} \det(A)^{g+1} f(A^{-1}x)$.

\subsection{Computing the $\G_m \times \PGL_2$-equivariant Chow ring of
  $\Aff_{sm}(2g+2)$ using $\G_m \times \GL_3$-counterparts}

In order to compute the Chow ring of the quotient stack $[\Aff_{sm}(2g
  + 2)/(\G_m \times \PGL_2)]$, we follow the method used by Di Lorenzo in
\cite[Section 2.3]{diL:18} by finding a $\G_m \times \GL_3$-counterpart of the
$\G_m \times \PGL_2$-scheme $\Aff_{sm}(2g + 2)$ to resolve the issue
that the group $\PGL_2$ is non-special. Let us briefly recall the
concepts from \cite{diL:18} here.

Given a $\PGL_2$-scheme $X$ of finite type over $\Spec k$, a
$\GL_3$-counterpart of $X$ is another scheme $Y$ equipped with a
$\GL_3$-action such that $[Y/\GL_3] \cong [X/\PGL_2]$. In
\cite[Theorem 1.4]{diL:18}, Di Lorenzo proves that a $\GL_3$-counterpart always
exists for a given $\PGL_2$-scheme. Indeed, with the adjoint
representation of $\PGL_2$, we have a morphism of algebraic groups $\PGL_2 \to \GL_3$. It follows
that
the quotient $\GL_3/\PGL_2$ is a $\GL_3$-counterpart of $\Spec
k$. Pulling back $[\Aff_{sm}(2g + 2)/\PGL_2]$ along the $\GL_3$-torsor
$\GL_3/\PGL_2 \to [\Spec k/\PGL_2] \cong \M_0$ will in turn give a
$\GL_3$-counterpart of the $\PGL_2$ action on $\Aff_{sm}(2g + 2)$. The $\G_m \times
\GL_3$-counterpart can be defined similarly.

Let $\Delta'$ denote the complement of $\Aff_{sm}(2g + 2)$ in $\Aff(2g + 2)$, \ie the discriminant locus of $\Aff(2g + 2)$. Because the weight of the action
of $\G_m$ in the quotient stack $[\Aff_{sm}(2g + 2)/(\G_m \times \PGL_2)]$ presenting $[\M_{0,2g + 2}/S_{2g + 2}]$ is different from the weight
of the $\G_m$ action that Di Lorenzo used in
his computation of $A^*(\H_g)$ for $g$ odd \cite{diL:18}, our next proposition
describes the $\G_m \times \GL_3$-counterpart for a $\G_m$ action
with an arbitrary weight. Precisely, if
$\lambda \cdot f = \chi(\lambda) f$ for some character $\chi$ of $\G_m$,
we consider the $\G_m \times \GL_3$-counterparts for $\Aff(2g + 2)$ and $\Delta'$, where $\G_m \times \PGL_2$
acts on $\Aff(2g+2)$ by
$$ (\lambda, A)\cdot f = \chi(\lambda) (\det A)^{g+1} f(A^{-1}x).$$

\begin{prop}\cite[Proposition 2.3 \& 2.6]{diL:18} \label{prop.dilorenzo}
Let $V_{g + 1}$ be the scheme parametrizing pairs $(q,f)$ where $q$ is a global section of $\mathcal{O}_{\Pro^2_S}(2)$ with zero locus $Q$, and $f$ is a global section of $\mathcal{O}_Q(g + 1)$. Let $D' \subset V_{g + 1}$ be the singular locus inside $V_{g + 1}$. Then
\begin{itemize}
\item $V_{g + 1}$ is a $\G_m \times \GL_3$-counterpart of $\Aff(2g + 2)$
  with the action of $\G_m \times \GL_3$ given by
\[(\lambda,A) \cdot (q,f) \coloneq (\det(A)q(A^{-1}x),\chi(\lambda) f(A^{-1}x)).\]  
    \item $D'$ is a $\G_m \times \GL_3$-counterpart of $\Delta'$.

\end{itemize}
\end{prop}

Applying Proposition \ref{prop.dilorenzo} we see that to compute
$A^*([\M_{0,2g+2}/S_{2g+2}])$ we must compute
$A^*_{\G_m \times \GL_3}(V_{g+1} \setminus D')$ where $\G_m \times \GL_3$
acts by $$(\lambda, A) \cdot (q,f) = (\det(A)q(A^{-1}x), \lambda^{-1}f(A^{-1}x)).$$

This computation is almost identical to the computation made by
Di Lorenzo in \cite[Section 4]{diL:18}.  The only difference is when
passing to the projectivization of $V_{g + 1}$ in order to trivialize
the global $\G_m$-action. Let $\sigma_0$ denote the
$\GL_3$-counterpart of $0 \in \Aff(2g + 2)$, which turns out to be the
zero section. The map $V_{g + 1} \setminus \sigma_0 \to \Pro(V_{g +
  1})$ is a $\G_m$-torsor over $\Pro(V_{g + 1})$ with the associated
line bundle $\mathcal{V}^{-1} \otimes \mathcal{O}(-1)$. Here
$\mathcal{V}$ is the standard character of $\G_m$. By contrast in \cite{diL:18}
the associated line bundle is ${\mathcal V}^{-2} \otimes {\mathcal O}(-1)$.

Following \cite{diL:18}, we denote by
$D$ the projectivization of $D'$.
The $\G_m$-torsor $V_{g+1} \setminus D' \to \Pro(V_{g + 1})\setminus D$
induces a surjective
pullback on the equivariant Chow groups $A^*_{\G_m \times \GL_3}(\Pro(V_{g
  + 1})\setminus D) \to A^*_{\G_m \times \GL_3}(V_{g + 1} \setminus
D')$ with kernel generated by $c_1(\mathcal{V}^{-1} \otimes
\mathcal{O}(-1)) = -\tau - h_{g + 1}$, where $\tau$ is the first Chern
class of the standard representation of $\G_m$, and $h_{g + 1}$ is the
first Chern class of $\mathcal{O}(1)$.

The same method used in the calculation of
$A^*(\H_g)$ for odd genus $g$ in \cite[Section 4]{diL:18} yields the following
result.
\begin{prop}
The integral Chow ring of $[(V_{g + 1}\setminus D')/(\G_m \times \GL_3)]$ is a quotient ring
\[\Z[\tau,c_2,c_3]/(2\tau^2 - 2g(g + 1)c_2,2(2g + 1)\tau,2c_3,p_{g + 1}(-\tau))\]
where $p_{g + 1}(-\tau) = p_{g + 1}(h_{g + 1})$ is the monic polynomial of degree $2g + 3$, which vanishes in $A^*_{\GL_3}(\Pro(V_{g + 1}))$, obtained by applying
the $\GL_3$-equivariant projective bundle theorem to $\Pro(V_{g + 1})$. 
\end{prop}

Finally, checking that the monic polynomial $p_{g + 1}(-\tau)$ is not in the ideal $(2\tau^2 - 2g(g + 1)c_2,2(2g + 1)\tau,2c_3)$ by using \cite[Proposition 6.5]{FuVi:11} completes the proof of Theorem \ref{thm.m2g+2}.

\section{Description of $\H_{g,w}$ and $[\M_{0,2g+2}/S_{2g+1}]$ as quotients
    by Borel subgroups of rank 2}
  In this section we give a description of the stacks $\H_{g,w}$ and $[\M_{0,2g+2}/S_{2g+1}]$ as quotients by Borel subgroups of $\GL_2/\mu_{g+1}$
 and $\GL_2/\mu_{2g+2}$ respectively.
  
Recall that Pernice \cite[Proposition 1.3]{Per:21} proved that the stack $\H_{g,1}$
is equivalent to the stack $\H'_{g,1}$ parametrizing the following data
$$(P \to S,\L, i \colon \L^2 \hookrightarrow {\mathcal O}_P, \sigma_P, j)$$
where $P \to S$ is a twisted $\Pro^1$, $\L$ is a line bundle on $P$ which restricts to a line bundle of degree $-(g+1)$ on the fibers of $P \to S$,
$\sigma_P \colon S \to P$ is a section and $j \colon \sigma_P^*\L \to {\mathcal O}_S$ satisfies $j^{\otimes 2} = \sigma_P^*(i)$.

Given this data Pernice follows \cite{ArVi:03} to obtain a family of pointed
hyperelliptic curves
by taking the double cover
$C = \Spec_{{\mathcal O}_P} ({\mathcal O}_P \oplus \L) \to P$ where the ${\mathcal O}_P$-algebra structure on ${\mathcal O}_P \oplus \L$ is induced by $i$.
If $f : C \to S$ is the corresponding family
of hyperelliptic curves, then Pernice shows that the pair $(\sigma_P, j)$ determines a section
$\sigma \colon S \to C$ such that $\sigma_P = f \circ \sigma$.

Now the section $\sigma$ is a Weierstrass section if and only if $\sigma$ is in the ramification divisor of the double cover
$C \to P$. This is equivalent to the condition that
$j=0$ as a map $\sigma_P^*\L \to {\mathcal O}_S$.
Putting this together we obtain 
the following description of $\H_{g,w}$.
\begin{prop}
  The stack $\H_{g,w}$ is equivalent to the stack $\H'_{g,w}$
  parametrizing the data
  $$(P \to S,\L, i \colon \L^2 \hookrightarrow {\mathcal O}_P, \sigma_P, 0).$$
\end{prop}

Let $\HH(2g+2,1) = \{(f,s)| f(0,1) = s^2\} \subset \Aff(2g+2) \times \Aff^1$
and let $\HH_{sm}(2g+2,1)$ be the open set where the form $f$ has distinct roots. Likewise let $\HH^w(2g + 2,1)$ be the divisor defined by the equation $s=0$,
and let $\HH^w_{sm}(2g+2,1)$ be its intersection with $\HH_{sm}(2g+2,1)$.
By \cite[Proposition 1.5]{Per:21}, the stack $\H'_{g,1}$
is equivalent to the quotient stack $[\HH_{sm}(2g+2,1)/(B_2/\mu_{g+1})]$
where $B_2$ is the Borel subgroup of lower triangular matrices in $\GL_2$,
and the action is given as follows. 
If $A= \begin{bmatrix} a & 0\\ b & c\end{bmatrix} \in B_2/\mu_{g+1}$ then 
$$A\cdot(f(x),s) =\left(f(A^{-1}x), c^{-(g+1)}s\right).$$
Note that our notation $\HH_{sm}(2g+2,1)$ differs from that of \cite{Per:21}.

The stack $\H_{g,w}$ is the quotient by $B_2/\mu_{g+1}$ of the divisor
  $\HH^w_{sm}(2g+2,1)$ defined by the equation
$s =0$ in $\HH_{sm}(2g+2,1)$.

By \cite[Proposition 4.4]{ArVi:03}, the group $G = B_2/\mu_{g+1}$ is isomorphic to $B_2$ if $g$ is even, and if $g$ is odd to
$\G_m \times (B_2/\G_m)$ where $B_2/\G_m$ is the group of lower triangular matrices in $\PGL_2$.
Under these identifications, the action of $B_2/\mu_{g+1}$ on $\HH_{sm}(2g+2,1)$
is given by 
\begin{itemize}
  \item 
If $g$ is even, $G = B_2$ acts by 
\[A \cdot (f(x),s) \coloneq 
\left((\det A)^{g}f(A^{-1}x), a^{{g\over{2}}}c^{-{{g+2}\over{2}}}s\right),\quad A = \begin{pmatrix} a & 0 \\ b & c 
\end{pmatrix}.\]
\item
  If $g$ is odd, $G = \G_m \times (B_2/\G_m)$ acts by
  \[ (\alpha, A) \cdot \left(f(x),s\right) \coloneq \left(\alpha^{-2} \det(A)^{g+1} f(A^{-1}x),\alpha^{-1} a^{{g+1\over{2}}} c^{-{g+1\over{2}}}s\right)\]
  where $A = \begin{bmatrix} a & 0\\ b & c\end{bmatrix}$ is a lower triangular matrix in $\PGL_2$.
    Note that the expression $\det(A)^{g+1} f(A^{-1}x)$ is invariant under homotheties so this gives a well-defined action of $\PGL_2$.
\end{itemize}

Now consider the quotient stack $[\M_{0,2g + 2}/S_{2g + 2 - n}]$. 

For any $n \geq 1$, the quotient stack $[\M_{0,2g+2}/S_{2g+2-n}]$
parametrizes the data $(P \to S, D_{2g+2}, \sigma_{1}, \ldots , \sigma_{n})$
where $P \to S$ is a twisted $\Pro^1$, $D_{2g+2} \subset P$ is an effective Cartier divisor which is finite and \'etale over $S$ of degree $2g+2$, and $\sigma_{1}, \ldots, \sigma_n$ are distinct sections of $D_{2g+2} \to S$.

The methods used above also realize the stack $[\M_{0,2g+2}/S_{2g+1}]$ as a quotient
  of $\HH_{sm}^w(2g+2,1)$ by $B_2/\mu_{2g+2}$.
  Recall that we identify $[\M_{0,2g+2}/S_{2g+1}]$ as the stack 
  parametrizing triples $(P \to S, D_{2g+2}, \sigma )$ where $P \to S$ is a twisted $\Pro^1$, $D_{2g+2} \subset P$ is finite and \'etale of degree $2g+2$ over $S$,
  and $\sigma \colon S \to D_{2g+2}$ is a section.
  Taking the marked point to be $(0:1) \in \Pro^1$ when rigidifying, we have
  an equivalence $[\M_{0,2g+2}/S_{2g+1}] = [\HH_{sm}^w(2g+2,1)/(B_2/\mu_{2g+2})]$.
  Since $2g+2$ is always even, the quotient $B_2/\mu_{2g+2}$ is always isomorphic to $\G_m \times (B_2/\G_m)$. In terms of this isomorphism, the action
  on $\HH_{sm}^w(2g+2,1)$ is given by 
  $$ (\alpha, A) \cdot f(x)  = \alpha^{-1} (\det A)^{g+1} f(A^{-1}x).$$

  \subsection{Description of the Chow rings of $\H_{g,w}$,
    $[\M_{0,2g+2}/S_{2g+1}]$,
    in terms of torus equivariant Chow groups}
  \noindent Since $B_2$ is an unipotent extension of the group of diagonal matrices
$D_2$, $B_2/\mu_{g+1}$ and $B_2/\mu_{2g+2}$
are unipotent extensions of the two-dimensional quotient tori
$D_2/\mu_{g+1}$ and $D_2/\mu_{2g+2}$ respectively. This implies
that for any algebraic space $X$, the map of stacks
$[X/(D_2/\mu_{g+1})] \to [X/(B_2/\mu_{g+1})]$
(\resp $[X/(D_2/\mu_{2g+2})] \to [X/(B_2/\mu_{2g+2})]$)
is an affine bundle. It follows that the flat pullback on equivariant
Chow groups induces an isomorphism
$A^*_{B_2/\mu_{g+1}}(X) \simeq A^*_{D_2/\mu_{g+1}}(X)$
(resp. $A^*_{B_2/\mu_{2g+2}}(X) \simeq A^*_{D_2/\mu_{2g+2}}(X)$).
Hence $A^*(\H_{g,1})=
A^*_T(\HH_{sm}(2g+2,1))$,
$A^*(\H_{g,w}) = A^*_T(\HH^w_{sm}(2g+2,1))$ and $A^*([\M_{0,2g+2}/S_{2g+1}]) =
A^*_T(\HH^w_{sm}(2g+2,1))$ for the actions of  the two-dimensional tori $T
=D_2/\mu_{g+1}$ and $T= D_2/\mu_{2g+2}$ respectively. We now describe these actions.

\begin{itemize}

\item If $g$ is even, then using coordinates $(t_0,t_1)$ on $T$, the action on
  $\HH(2g+2,1)$ is given by
    $$(t_0,t_1)\cdot \left( f(x_0,x_1),s\right) = \left((t_0t_1)^gf(x_0/t_0, x_1/t_1),t_0^{{g\over{2}}}t_1^{-{g+2\over{2}}}s\right).$$
    
  \item If $g$ is odd, then using coordinates $(\alpha, \rho)$ on $T$,
the action on $\HH(2g+2,1)$ is given by
$$(\alpha, \rho) \cdot \left(f(x_0,x_1),s \right) =
\left(\alpha^{-2} \rho^{g+1} f(x_0/\rho, x_1), \alpha^{-1} \rho^{{g+1\over{2}}}s\right).$$

  \item For the stack $[\M_{0,2g+2}/S_{2g+1}]$, using coordinates
    $(\alpha, \rho)$ on $T$, the action on $\HH^{w}(2g+2,1)$ is given by 
    $$(\alpha, \rho) \cdot f(x_0,x_1) = \alpha^{-1} \rho^{g+1} f(x_0/\rho, x_1).$$
\end{itemize}

Consider the inclusion $\Aff(2g+1) \hookrightarrow \Aff(2g+2)$, $h \mapsto x_0h$.
The variety $\HH^w_{sm}(2g+2,1)$ can be identified with the open set
$U_w \coloneq (\Aff(2g+1) \setminus \Delta') \setminus L$ where $\Delta'$ is the discriminant locus of the space $\Aff(2g + 1)$ as defined in the previous section, and $L$ is the linear subspace of forms satisfying $h(0,1) =0$.
By choosing a suitable action of the rank-two torus $T$ on
$\Aff(2g+1)$, we can ensure that the map
$\Aff(2g+1) \to \Aff(2g+2)$ is $T$-equivariant for the various actions
of $T$ on $\Aff(2g+2)$.
In particular, we may calculate the Chow rings of the stacks $\H_{g,w}$
and $[\M_{0,2g+2}/S_{2g+1}]$ as the $T$-equivariant Chow rings of
$U_w$ for different actions of the torus $T$. This is summarized in the following proposition.

\begin{prop}\label{action-1pt}  \hspace*{\fill}
  \begin{itemize}
  \item If $g$ is even, then $A^*(\H_{g,w}) = A^*_T(U_w)$ where
    the action of $T$ is given by  $$(t_0,t_1) \cdot h(x_0,x_1) = t_0^{g-1} t_1^g h(x_0/t_0,x_1/t_1).$$
  \item If $g$ is odd, then $A^*(\H_{g,w}) = A^*_T(U_w)$ where
    the action of $T$ is given by $$(\alpha, \rho) \cdot h(x_0,x_1) =
    \alpha^{-2} \rho^g h(x_0/\rho, x_1).$$
  \item For any genus $g$, $A^*([\M_{0,2g+2}/S_{2g+1}]) = A^*_T(U_w)$
    where $T$ acts by $$(\alpha, \rho) \cdot h(x_0,x_1) = \alpha^{-1} \rho^{g} h(x_0/\rho, x_1).$$
  \end{itemize}
\end{prop}

\section{Description of $\H_{g,2}^w, \H_{g,3}^w$, $[\M_{0,2g+2}/S_{2g}]$,
  $[\M_{0,2g+2}/S_{2g-1}]$ as quotients by tori}

\subsection{The 2-Weierstrass pointed case} \label{sec.2weierstrass}
Let $\H_{g,2}^o$ be the open substack of $\H_{g,2}$ parametrizing
two-pointed hyperelliptic curves $(C \to S, \sigma_1, \sigma_2)$ such
that for every $s \in S$, $\sigma_1(s) + \sigma_2(s)$ does not equal to
the $g^1_2$.
The construction of \cite{Per:21} can be extended to
give a presentation for $\H_{g,2}^o$ as a quotient stack. The
stack $\H_{g,2}^o$ 
is equivalent to the stack $({\H^{o}_{g,2}})'$ parametrizing the following data
$$(P \to S,\L, i \colon \L^2 \hookrightarrow {\mathcal O}_P, \sigma_{P,1} , \sigma_{P,2}, j_1, j_2)$$
where $P \to S$ is a twisted $\Pro^1$, $\L$ is a line bundle on $P$ which restricts to a line bundle of degree $-(g+1)$ on the fibers of $P \to S$,
$\sigma_{P,1}, \sigma_{P,2} \colon S \to P$ are distinct sections and $j_\ell \colon \sigma_{P,\ell}^*\L \to {\mathcal O}_S$ satisfies $j_{\ell}^{\otimes 2} = \sigma_{P,\ell}^*(i)$ for $\ell = 1,2$.

The stack $\H_{g,2}^w$ is equivalent to the closed substack $(\H^{w}_{g,2})'$ in $(\H^{o}_{g,2})'$
parametrizing tuples where both $j_1, j_2$ are the zero map.
The same argument of \cite[Proposition 1.5]{Per:21} shows that 
the stack $(\H^{o}_{g,2})'$ can be realized as the quotient stack
$[\HH_{sm}(2g+2,2)/(D_2/\mu_{g+1})]$. Here
\[ \HH_{sm}(2g+2,2) = \{ (f,s_0,s_\infty)| f(0,1) = s_0^2, f(1,0) = s_\infty^2\}
\subset \Aff_{sm}(2g+2) \times \Aff^2\]
and $D_2/\mu_{g+1}$ is the torus in $\GL_2/\mu_{g+1}$ that fixes $(1:0)$ and $(0:1)$
in $\Pro^1$.
The action of $D_2/\mu_{g+1}$ is given by
\[A \cdot \left( f(x), s_0, s_\infty\right) = \left(f(A^{-1}x), c^{-(g+1)}s_0, a^{-(g+1)}s_\infty\right),\quad A = \begin{bmatrix}
a & 0 \\ 0 & c
\end{bmatrix}.\]
With this identification $\H_{g,2}^w$ is the quotient by $D_2/\mu_{g+1}$
of the codimension-two subvariety
$\HH_{sm}^w(2g+2,2) = \{ f \in \Aff_{sm}(2g+2)| f(0,1) = f(1,0) = 0\} \subset
\Aff_{sm}(2g+2)$.

A similar analysis shows that $[\M_{0,2g+2}/S_{2g}]$ is equivalent to the quotient
stack $[\HH_{sm}^w(2g+2,2)/(D_2/\mu_{2g+2})]$.
Since the quotient $D_2/\mu_{g+1}$ (resp. $D_2/\mu_{2g+2}$) is a rank-two torus $T,$ we may compute $A^*(\H_{g,2}^o)$, $A^*(\H_{g,2}^w)$ (resp. $A^*([\M_{0,2g+2}/S_{2g}])$), by computing the $T$-equivariant Chow groups. The $T$-actions are described as follows. 

\begin{itemize}
\item If $g$ is even, then using coordinates $(t_0,t_1)$ on $T$, the action on
  $\HH(2g+2,2)$ is given by
  $$(t_0,t_1)\cdot \left(f(x_0,x_1),s_0, s_\infty\right) = \left(
  (t_0t_1)^g f(x_0/t_0, x_1/t_1),t_0^{{g\over{2}}}t_1^{-{g+2\over{2}}}s_0, t_0^{-{g+2\over{2}}}t_1^{{g\over{2}}}s_\infty\right).$$
  
  \item If $g$ is odd, then using coordinates $(\alpha, \rho)$ on $T$,
the action on $\HH(2g+2,2)$ is given by
$$(\alpha, \rho) \cdot \left(f(x_0,x_1),s_0,s_\infty\right) =
\left(\alpha^{-2} \rho^{g+1} f(x_0/\rho, x_1),
\alpha^{-1} \rho^{{g+1\over{2}}}s_0, \alpha^{-1}\rho^{-{g+1\over{2}}}s_\infty\right).$$

  \item For the stack $[\M_{0,2g+2}/S_{2g}]$, using coordinates
    $(\alpha, \rho)$ on $T$ the action on $\HH^{w}(2g+2,2)$ is given by 
    $$(\alpha, \rho) \cdot f(x_0,x_1) = \alpha^{-1} \rho^{g+1} f(x_0/\rho, x_1).$$
\end{itemize}

As in the one-pointed case we can consider the map
$\Aff(2g) \to \Aff(2g+2)$, $h \mapsto x_0x_1 h$, and we see that
$\HH^w_{sm}(2g+2,2)$ is the image of the open set $U_{w,2} \subset \Aff(2g)$
where $U_{w,2} \coloneq (\Aff(2g) \setminus \Delta') \setminus (L_0 \cup L_\infty)$
is the open set parametrizing forms $h(x_0,x_1)$ of degree $2g$ with distinct roots such that both $h(1,0)$ and $h(0,1)$ are non-zero.

For suitable choices of $T$-action on $\Aff(2g)$, we can make the map
$U_{w,2} \to \HH_{sm}^w(2g+2,2)$ $T$-equivariant with respect to the various
actions on $\HH_{sm}^w(2g+2,2)$.
This can be summarized as follows.
\begin{prop}\label{action-2pt}\hspace*{\fill}
  \begin{itemize}
  \item If $g$ is even, then $A^*(\H_{g,2}^w) = A^*_T(U_{w,2})$ where the action
    of $T$ is given by $$(t_0,t_1) \cdot h(x_0,x_1) = (t_0t_1)^{g-1} h(x_0/t_0,x_1/t_1).$$
  \item If $g$ is odd, then $A^*(\H_{g,2}^w) = A^*_T(U_{w,2})$ where
    the action of $T$ is given by $$(\alpha, \rho) \cdot h(x_0,x_1) =
    \alpha^{-2} \rho^{g} h(x_0/\rho, x_1).$$
  \item For any genus $g$, $A^*([\M_{0,2g + 2}/S_{2g}]) = A^*_T(U_{w,2})$
    where $T$ acts by $$(\alpha, \rho) \cdot h(x_0,x_1) = \alpha^{-1} \rho^{g} h(x_0/\rho, x_1).$$
  \end{itemize}
\end{prop}

\subsection{The $3$-Weierstrass-pointed case}
A similar construction can be used to present
  $\H_{g,3}^w$ as closed quotient substack of the stack $\H_{g,3}^o$ parametrizing 3-pointed hyperelliptic curves where no two of the points sum to a $g^1_2$.
  In this case $\H_{g,3}^o = [\HH_{sm}(2g+2,3)/ (Z(\GL_2/\mu_{g+1}))]$ where
    $\HH_{sm}(2g+2,3) = \{ (f,s_0,s_1, s_\infty) | f(0,1) = s_0^2, f(1,1) = s_1^2, f(1,0) = s_\infty^2\} \subset \Aff_{sm}(2g+2) \times \Aff^3$ and $Z(\GL_2/\mu_{g+1}) = \G_m$
    is the center of $\GL_2/\mu_{g+1}$. Similarly, we also have that 
    $[\M_{0,2g + 2}/S_{2g - 1}]$ is equivalent to the quotient stack $[\HH_{sm}^w(2g + 2,3)/\G_m]$ where $\HH_{sm}^w(2g + 2,3) = \{f \in \Aff_{sm}(2g + 2) \mid f(0,1) = f(1,1) = f(1,0) = 0\}$ is a codimension-three subvariety in $\Aff_{sm}(2g + 2)$. The $\G_m$-action is described as follows.
    \begin{itemize}
    \item If $g$ is even, then the action on $\HH(2g + 2,3)$ is given by
      \begin{align*}
    	t \cdot (f(x_0,x_1),s_0,s_1,s_\infty) &= 
        \left(t^{2g}f(x_0/t,x_1/t), t^{-1}s_0, t^{-1}s_1, t^{-1}s_\infty \right)\\
        &= \left(t^{-2}f(x_0,x_1), t^{-1}s_0, t^{-1}s_1, t^{-1}s_\infty \right).
        \end{align*}
    \item If $g$ is odd, then the action on $\HH(2g + 2,3)$ is given by
    	\[t \cdot (f(x_0,x_1),s_0,s_1,s_\infty) = \left(t^{-2}f(x_0,x_1), t^{-1}s_0, t^{-1}s_1, t^{-1}s_\infty \right).\]
    \item For the stack $[\M_{0,2g + 2}/S_{2g - 1}]$, the action on $\HH^w(2g +2,3)$ is given by
    	\[t \cdot f(x_0,x_1) = t^{-1}f(x_0,x_1).\]
    \end{itemize}
    
    The stack $\H_{g,3}^w$ is the quotient of the closed codimension-three subset where all of the three $s_i$'s are 0. This in turn can be realized as
    $[U_{w,3}/\G_m]$ where $U_{w,3} \coloneq (\Aff(2g-1) \setminus \Delta') \setminus (L_0 \cup L_1 \cup L_\infty)$ is the parameter space of binary forms of degree $2g-1$ with distinct roots which do not vanish at $(0:1), (1:1), (1:0)$ in $\Pro^1$, with the induced action of $\G_m$ described in the following proposition.

\begin{prop}\label{action-3pt}\hspace*{\fill}
	\begin{itemize}
		\item For any genus $g$, $A^*(\H_{g,3}^w) = A^*_{\G_m}(U_{w,3})$ where the action of $\G_m$ is given by $$t \cdot h(x_0,x_1) = t^{-2}h(x_0,x_1).$$
		\item For any genus $g$, $A^*([\M_{0,2g + 2}/S_{2g - 1}]) = A^*_{\G_m}(U_{w,3})$ where $\G_m$ acts by $$t \cdot h(x_0,x_1) = t^{-1}h(x_0,x_1).$$
	\end{itemize}
\end{prop}

    \begin{remark}
      The integral Chow ring of $\H_{g,3}^w$ can also be computed using Proposition \ref{prop.danpetersen}. However, the advantage of working with the presentation $[U_{w,3}/\G_m]$ is that it will allow us to identify a natural generator
      for $\Pic(\H_{g,3}^w)$.
      \end{remark}

\section{Computing the Chow rings}
In this section we prove Theorem \ref{thm.Hgw} and \ref{thm.Hg2w} by computing the appropriate $T$-equivariant Chow rings as indicated in Proposition \ref{action-1pt} and \ref{action-2pt}. 
We begin with a general
result on the $T$-equivariant Chow ring of the set of binary forms of degree $N$
with distinct roots.

\subsection{The equivariant Chow ring of $\Aff(N)\setminus \Delta'$}\label{general-result}\hspace*{\fill}

\noindent Suppose that we are given an arbitrary action of $T = (\G_m)^2$ on $\Aff(N)$, the affine space of binary forms of degree $N$ in $x_0,x_1$, by
\begin{align}\label{action}
    (t_0,t_1) \cdot f(x_0,x_1) = t_0^a t_1^b f(t_0^{-\alpha_0} t_1^{-\alpha_1} x_0, t_0^{-\beta_0} t_1^{-\beta_1} x_1).
\end{align}
The goal of this section is to give a presentation for the equivariant
Chow ring $A^*_T(\Aff(N) \setminus \Delta')$ where $\Delta'$ is the
discriminant locus of $\Aff(N)$ as defined before, in terms of the
weights of the action.

Let $\lambda_1 : T \to \G_m, (t_0,t_1) \mapsto t_0$ and $\lambda_2 : T \to \G_m, (t_0,t_1) \mapsto t_1$ be the standard generators for the character group of the torus $T$. Set $c_1(\lambda_1) = l_1$ and $c_1(\lambda_2) = l_2$ in $A^*(BT)$. Define the characters $\chi_1,\chi_2$ by $\chi_1(t) = t_0^{\alpha_0}t_1^{\alpha_1}, \chi_2(t) = t_0^{\beta_0}t_1^{\beta_1}$, and set
\begin{align*}
    T_1 &\coloneq c_1(\chi_1) = c_1(\lambda_1^{\alpha_0}\lambda_2^{\alpha_1}) = \alpha_0 l_1+ \alpha_1 l_2, \\
    T_2 &\coloneq c_1(\chi_2) = c_1(\lambda_1^{\beta_0}\lambda_2^{\beta_1}) = \beta_0 l_1 + \beta_1 l_2.
\end{align*}

Let $\B(N) = \Pro(\Aff(N))$ be the projectivization of the affine space $\Aff(N)$ of binary forms of degree $N$.
We denote by $\Delta$ the image of the discriminant locus $\Delta'$ under the projectivization map.

\begin{prop} \label{prop.disc}
The equivariant Chow ring $A^*_T(\Aff(N) \setminus \Delta')$ is a quotient ring
\[\Z[l_1,l_2,\xi]/(\alpha_{1,0}(\xi),\alpha_{1,1}(\xi),al_1 + bl_2 - \xi,p(\xi))\]
where $\xi = c_1(\mathcal{O}_{\B(N)}(1))$ is the hyperplane class and
\begin{align*}
    \alpha_{1,0} &= 2(N - 1)\xi - N(N - 1)(T_1 + T_2) \\
                 &= 2(N - 1)\xi - N(N - 1)[(\alpha_0 + \beta_0)l_1 + (\alpha_1 + \beta_1)l_2], \\
    \alpha_{1,1} &= \xi^2 - (T_1 + T_2)\xi - N(N - 2)T_1T_2 \\
                 &= \xi^2 - [(\alpha_0 + \beta_0)l_1 + (\alpha_1 + \beta_1)l_2]\xi - N(N - 2)(\alpha_0l_1 + \alpha_1l_2)(\beta_0l_1 + \beta_1l_2), \\
    p(\xi) &= \prod_{i = 0}^N [\xi - (N - i)T_1 - iT_2] \\
           &= \prod_{i = 0}^N [\xi - (N - i)(\alpha_0l_1 + \alpha_1l_2) - i(\beta_0l_1 + \beta_1l_2)].
\end{align*}
\end{prop}

\begin{proof}
Choose coordinates $(X_0 : X_1 : \ldots : X_N)$ on $\B(N)$ so that the coordinate function $X_i$ is the coefficient of $x_0^{N - i}x_1^i$ in a binary form of degree $N$. Then there is an induced action of $T$ on $\B(N)$ given by
\[t\cdot (X_0 : \ldots : X_N) = (\chi_1^{-N}(t)X_0 : \ldots : \chi_1^{-(N - i)}(t)\chi_2^{-i}(t)X_i : \ldots : \chi_2^{-N}(t)X_N).\]

Following \cite[Lemma 2.3]{EdFu:07} which is proved for $\G_m$-actions in \cite[Section 3.3]{EdGr:98}, we obtain that $A^*_T(\B(N)) = \Z[l_1,l_2,\xi]/p(\xi)$ where $\xi = c_1(\mathcal{O}_{\B(N)}(1))$ is the hyperplane class, and $p(\xi) = \prod_{i = 0}^N (\xi - (N - i)T_1 - iT_2)$ is a degree $N + 1$ monic polynomial since $\{-(N - i)T_1 - iT_2\}_{i = 0}^N$ is the set of equivariant Chern roots of the representation $\Aff(N)$.

By \cite[Proposition 4.2]{EdFu:07} the ideal generated by the image of $A^*_T(\Delta) \to A^*_T(\B(N))$ is computed, and we get
\[A^*_T(\B(N) \setminus \Delta) = A^*_T(\B(N))/(\alpha_{1,0}(\xi),\alpha_{1,1}(\xi)).\]

Now we have a $\G_m$-torsor $\Aff(N) \setminus \Delta' \to \B(N) \setminus \Delta$ associated to the line bundle $\lambda_1^{\otimes a} \otimes \lambda_2^{\otimes b} \otimes \mathcal{O}(-1)$. An argument similar to the one used in the proof of \cite[Lemma 3.2]{EdFu:07} shows that the pullback $A^*_T(\B(N) \setminus \Delta) \to A^*_T(\Aff(N) \setminus \Delta')$ is surjective with kernel generated by $al_1 + bl_2 - \xi$. Substituting $\xi = al_1 + bl_2$ yields the presentation of the proposition.
\end{proof}

\subsection{Integral Chow ring of $\H_{g,n}^w$ for even genus $g$ and $n = 1,2$}\hspace*{\fill}

\noindent In this section, we will apply the results
of the previous section to compute the integral Chow rings of $\H_{g,n}^w$ for even genus $g$ and $n = 1,2$.

\subsubsection{$1$-Weierstrass-pointed case}

Recall that by Proposition \ref{action-1pt} if the genus $g$ is even,
then $A^*(\H_{g,w})=A^*_T(U_w)$ where $U_w = (\Aff(2g + 1)\setminus \Delta') \setminus L$ with the action of $T = \G_m^2$ on $\Aff(2g + 1)$ defined by
\[(t_0,t_1)\cdot f(x_0,x_1) = t_0^{g - 1}t_1^{g}f(x_0/t_0,x_1/t_1).\]
Then by Proposition \ref{prop.disc} with
$N = 2g + 1$, $a = g - 1, b = g$ and $(\alpha_0,\alpha_1) = (1,0), (\beta_0,\beta_1) = (0,1)$ we have 
\begin{align*}
A_T^*(\Aff(2g + 1) \setminus \Delta') &= \dfrac{\Z[l_1,l_2][\xi]}{\left(\alpha_{1,0}(\xi),\alpha_{1,1}(\xi),(g-1)l_1 + gl_2 - \xi, p(\xi)\right)}
\end{align*}
where
\[p(\xi) = \prod_{i = 0}^{2g + 1} [\xi - (2g - i + 1)l_1 - il_2].\]

Let $i : L \setminus (\Delta' \cap L) \hookrightarrow \Aff(2g + 1)\setminus \Delta'$ be the $T$-equivariant closed embedding.
By the localization sequence for equivariant Chow groups,
$A^*_T( (\Aff(2g+1) \setminus \Delta') \setminus L)$ is the quotient
of $A^*_T(\Aff(2g+1)) \setminus \Delta')$ by 
$$\im \left(i_* \colon  A^*_T(L \setminus (\Delta' \cap L)) \to A^*_T(\Aff(2g+1) \setminus \Delta')\right).$$
Since $L$ and $\Aff(2g+1)$ are linear subspaces, the pullback $i^* \colon A^*_T(\Aff(2g+1) \setminus \Delta')\to
A^*_T(L \setminus (\Delta' \cap L))$ is necessarily surjective because it factors
the surjective map $A^*_T(\Spec k) \to A^*_T(L \setminus (\Delta' \cap L))$
through the surjection $A^*_T(\Spec k) \to A^*_T(\Aff(2g+1) \setminus \Delta')$.
By the projection formula $i_*i^*[\alpha] = [\alpha] \cdot [L \setminus (\Delta' \cap L)]$ for any class $[\alpha] \in A^*(\Aff(2g+1)\setminus \Delta')$. Thus
the image of $i_*$ is generated by the image of the fundamental class of the hyperplane $[L]_T$ under the composite $i_* j^*$ where $j \colon L \setminus
(\Delta' \cap L) \hookrightarrow L$ is the open immersion.

Thus to finish obtaining a presentation for $A^*_T(U_w)$ we need
only compute the $T$-equivariant fundamental class of the hyperplane
$L$. This hyperplane is defined by the vanishing of the coefficient of $x_1^{2g+1}$
in a binary form and its $T$-equivariant fundamental class is
$$[L]_T = c_1\left(\lambda_1^{g - 1}\lambda_2^{g - (2g + 1)}\right) = (g-1)l_1 - (g+1)l_2.$$
Thus we get
$$
 A_T^*(\mathbb{A}(2g + 1) \setminus \Delta' \setminus L) =
 \dfrac{\Z[l_1,l_2]}{\left(\alpha_{1,0}(\xi),\alpha_{1,1}(\xi), p(\xi),
    (g-1)l_1-(g+1)l_2\right)}$$
where
\begin{align*}
& \xi = (g-1)l_1 + gl_2, \\
& \alpha_{1,0}(\xi) = -\left[2g(l_1 + l_2) + 4gl_1\right], \\
& \alpha_{1,1}(\xi) = (g - 1)(g - 2)l_1^2 + g(g - 1)l_2^2 - 2(g^2 + 2g - 1)l_1l_2.
\end{align*}
\begin{prop}
The polynomial $p((g - 1)l_1 + gl_2)$ is in the ideal generated by $\alpha_{1,0}((g - 1)l_1 + gl_2)$, $\alpha_{1,1}((g - 1)l_1 + gl_2)$ and $(g - 1)l_1 - (g + 1)l_2$.
\end{prop}
\begin{proof}
We can write
\[p((g - 1)l_1 + gl_2) = \prod_{i = 0}^{2g + 1}\left[(-g + i - 2)l_1 + (g - i)l_2\right].\]
Notice that the $(i = 2g + 1)$-th term on the right hand side of the above equality is
\[(g - 1)l_1 + (-g - 1)l_2,\]
thus it is contained in the ideal generated by $(g - 1)l_1 - (g + 1)l_2$.
\end{proof}

\begin{prop}
$\alpha_{1,1}((g - 1)l_1 + gl_2)$ is in the ideal generated by $\alpha_{1,0}((g - 1)l_1 + gl_2)$ and $(g - 1)l_1 - (g + 1)l_2$. 
\end{prop}
\begin{proof}
For simplicity, let us denote by $g_1,g_2$ the two degree one generators $\alpha_{1,0}((g - 1)l_1 + gl_2)$ and $(g - 1)l_1 - (g + 1)l_2$ respectively.
It is easy to see that
\[\alpha_{1,1}((g - 1)l_1 + gl_2) = (l_2)\cdot g_1 + [(g - 2)l_1 - gl_2]\cdot g_2. \qedhere\]
\end{proof}

Therefore for even genus $g$,
\[A^*(\H_{g,w}) = \dfrac{\Z[l_1,l_2]}{(6gl_1 + 2gl_2,(g - 1)l_1 - (g + 1)l_2)}.\]
Note that there is a natural ring homomorphism $A^*(\H_g) \to A^*(\H_{g,w})$, and it can be given explicitly by sending $c_1 \mapsto l_1 + l_2$ and $c_2 \mapsto l_1l_2$. More precisely, 
\[\begin{tikzcd}[row sep = huge]
A^*(\H_{g}) = \dfrac{\Z[c_1,c_2]}{(2(2g + 1)c_1,g(g - 1)c_1^2 - 4g(g + 1)c_2)} \arrow[d,"\substack{c_1 \mapsto l_1 + l_2 \\ c_2 \mapsto l_1l_2}"] \\
A^*(\H_{g,w}) = \dfrac{\Z[l_1,l_2]}{(6gl_1 + 2gl_2,(g - 1)l_1 - (g + 1)l_2)}.
\end{tikzcd}\]
We can check that the images of the generators of ideal for $A^*(\H_g)$ are indeed vanishing in the Chow ring $A^*(\H_{g,w})$. Namely, let $f_1(c_1) = 2(2g + 1)c_1$ and $f_2(c_1,c_2) = g(g - 1)c_1^2 - 4g(g + 1)c_2$, then
\begin{align*}
f_1(l_1 + l_2) &= 2(2g + 1)(l_1 + l_2) \\
&= 1 \cdot g_1 + (-2) \cdot g_2,
\end{align*}
and
\begin{align*}
f_2(l_1+l_2,l_1l_2) &= g(g - 1)(l_1 + l_2)^2 - 4g(g + 1)l_1l_2 \\
&= (-l_2)\cdot g_1 + [g(l_1 - l_2)]\cdot g_2,
\end{align*}
where $g_1,g_2$ denote two degree one generators of ideal of $A^*(\H_{g,w})$ as above.

Furthermore, the above ring map is injective in codimension one. In order to see this, for integers $A,B,a \in \Z$ such that
\[A(6gl_1 + 2gl_2) + B[(g - 1)l_1 - (g + 1)l_2] = a(l_1 + l_2), \]
we need to prove that the minimum integer $a$ with the above equality holds has to be $2(2g + 1)$. It can be seen by comparing the coefficients of $l_1,l_2$ on both sides, so we have
\[A = \dfrac{a}{2(2g + 1)},\quad B = \dfrac{-a}{2g + 1},\]
and thus it implies the ring homomorphism is injective on Picard groups.

In particular, by computing the Smith normal form of the two degree one relations, we can write
\begin{align*}
A^*(\H_{g,w}) &= \dfrac{\Z[l_1,l_2]}{(l_2 + (4g^2 - 2g + 1)l_1, 4g(2g + 1)l_1)} \\
&\cong \dfrac{\Z[l_1]}{(4g(2g + 1)l_1)}
\end{align*} 
which means the integral Picard group $\Pic(\H_{g,w})$ for even genus $g$ is cyclic of order $4g(2g + 1)$. \\

{\noindent\bf Natural generator for $A^*(\H_{g,w})$.}
There are multiple integral linear combinations of $l_1, l_2$
which generate the integral Picard group of $\H_{g,w}$. We conclude our proof
by showing that the
restriction to $\H_{g,w}$ of $\psi$-class on $\H_{g,1}$ generates $A^*(\H_{g,w})$.
\begin{cor}
  $A^*(\H_{g,w})$ is generated by $\psi$, and consequently
  $$A^*(\H_{g,w}) = \Z[\psi]/(4g(2g+1)\psi).$$
\end{cor}
\begin{proof}

According to the degree one relations, we can rewrite these two relations as follows
\[\begin{cases}
(l_1 - l_2) + (-2)\left[\left(1 - \dfrac{g}{2}\right)l_1 + \dfrac{g}{2}l_2\right] = 0, \\
4g(2g + 1)\left[\left(1 - \dfrac{g}{2}\right)l_1 + \dfrac{g}{2}l_2\right] = 0.
\end{cases}\]
By setting
\begin{align*}
l_3 &\coloneq \left[\left(1 - \dfrac{g}{2}\right)l_1 + \dfrac{g}{2}l_2\right] + \underbrace{\left[(g - 1)l_1 - (g + 1)l_2\right]}_{\text{relation}} \\
&= \dfrac{g}{2}l_1 - \dfrac{g + 2}{2}l_2,
\end{align*}
we get a new generator of $\Pic(\H_{g,w})$ with order $4g(2g + 1)$.

{\bf Claim.} The pullback of $\psi$ to $\H_{g,w}$ has Chow class
$-{g\over{2}}l_1 + {g+2\over{2}}l_2$.

{\bf Proof of Claim.} The restriction of $\psi$ to $\H_{g,w}$
is the Chow class which associates to the Weierstrass-pointed hyperelliptic
curves $(C \to S, \sigma_w)$ the class $\sigma_w^*(c_1({\mathcal O}(-\sigma_w)))$. On the
one hand, we may also consider the class $c_1({\mathcal O}_{\H_{g,1}}(\H_{g,w}))$. Its restriction to a family $(C \to S, \sigma_w)$ is the class
$\sigma_w^*(c_1({\mathcal O}(\sigma_w))) = -\psi$.
On the other hand, we know that $\H_{g,w}$ is defined by the equation
$s=0$ in the quotient stack $[\HH_{sm}(2g+2,1)/B_2]$. This equation
has $T$-weight $t_0^{g\over{2}} t_1^{-{g+2\over{2}}}$. Thus,
$c_1({\mathcal O}_{\H_{g,1}}(\H_{g,w})) = {g\over{2}}l_1 - {g+2\over{2}}l_2
\in A^*_T(\HH_{sm}^w(2g + 2,1)) = A^*(\H_{g,w})$. 
The claim now follows.
\end{proof}

\subsubsection{$2$-Weierstrass-pointed case}
By Proposition \ref{action-2pt}, $A^*(\H_{g,2}^w) = A^*_T(U_{w,2})$
where $U_{w,2} = (\Aff(2g) \setminus \Delta') \setminus (L_0 \cup L_\infty)$
and $T$ acts on $\Aff(2g)$ by
$$(t_0,t_1) \cdot h(x_0,x_1) = (t_0t_1)^{g-1} h(x_0/t_0, x_1/t_1).$$

By Proposition \ref{prop.disc} with 
 $N =2g$, $a = b = g - 1$ and $(\alpha_0,\alpha_1) = (1,0)$, $(\beta_0,\beta_1) = (0,1)$ we have that
\[A_T^*(\Aff(2g)\setminus \Delta') = \dfrac{\Z[l_1,l_2][\xi]}{(\alpha_{1,0}(\xi),\alpha_{1,1}(\xi),p(\xi),(g - 1)(l_1 + l_2) - \xi)}\]
where 
\begin{align*}
& p((g - 1)(l_1 + l_2)) = \prod_{i = 0}^{2g}[(-g + i - 1)l_1 + (g - i - 1)l_2],\\
& \alpha_{1,0}((g - 1)(l_1 + l_2)) = -2(2g - 1)(l_1 + l_2), \\
& \alpha_{1,1}((g - 1)(l_1 + l_2)) = (g - 1)(g - 2)(l_1 + l_2)^2 - 4g(g - 1)l_1l_2.
\end{align*}
The hyperplanes $L_0,L_\infty$ are defined by the vanishing of the
coefficients of $x_1^{2g},x_0^{2g}$ respectively, and their equivariant fundamental classes are
$$[L_0]_T = (g - 1)l_1 - (g + 1)l_2,$$
and
$$[L_\infty]_T =  -(g + 1)l_1 + (g - 1)l_2.$$

By a similar argument used in the one-pointed case, we conclude
that
\[A_T^*(U_{w,2}) = \dfrac{\Z[l_1,l_2]}{(\alpha_{1,0}(\xi),
  \alpha_{1,1}(\xi),p(\xi),(g-1)l_1 -(g+1)l_2, -(g+1)l_1 + (g-1)l_2)}\]
where $\xi = (g-1)(l_1 + l_2)$.

\begin{prop}
The polynomials $p((g - 1)(l_1 + l_2))$, $\alpha_{1,0}((g - 1)(l_1 + l_2))$ and the degree two relation $\alpha_{1,1}((g - 1)(l_1 + l_2))$ are all in the ideal generated by the classes of two hyperplanes $[L_0]_T,[L_\infty]_T$.
\end{prop}
\begin{proof}
Notice that $[L_0]_T + [L_\infty]_T = -2(l_1 + l_2)$ so $\alpha_{1,0}((g-1)(l_1 + l_2) = (2g-1)([L_0]_T + [L_\infty]_T$. 
  The first term in the product of $p((g - 1)(l_1 + l_2))$ associated to $i = 0$ is $-(g + 1)l_1 + (g - 1)l_2 = [L_0]_T$. 
   To see that  $\alpha_{1,1}$ is in the ideal generated by $[L_0]_T, [L_\infty]_T$  note that 
\begin{align*}
&\phantom{{}={}}\alpha_{1,1}((g - 1)(l_1 + l_2)) \\
&= (g - 1)(g - 2)(l_1 + l_2)^2 - 4g(g - 1)l_1l_2 \\
&= \left[\dfrac{(g - 1)(g - 2)}{2}\right]2(l_1 + l_2)^2 + (g - 1)^2l_2[L_\infty]_T + (g^2 - 1)l_2[L_0]_T. \qedhere
\end{align*}
\end{proof}

The integral Chow ring of $H_{g,2}^w$ can be expressed as
\[A^*(\H_{g,2}^w) = \dfrac{\Z[l_1,l_2]}{((g - 1)l_1 - (g + 1)l_2,2(l_1 + l_2))}.\]

In particular, the Picard group of $\H_{g,2}^w$ can be computed as
\begin{align*}
\Pic(\H_{g,2}^w) &= \dfrac{\Z l_1 \oplus \Z l_2}{((g - 1)l_1 - (g + 1)l_2,2(l_1 + l_2))} \cong \dfrac{\Z l_2}{(4g l_2)} \cong \dfrac{\Z l_1}{(4g l_1)},
\end{align*}
meaning that the integral Picard group is cyclic of order $4g$. To obtain a natural generator of the Picard group  we can rewrite the degree one relations as
\[\begin{cases}
-2\left(l_1 + \dfrac{g}{2}(l_2 - l_1)\right) + (l_1 - l_2) = 0, \\
4\left(l_1 + \dfrac{g}{2}(l_2 - l_1)\right) + 2(g - 1)(l_1 - l_2) = 0.
\end{cases}\]
Thus
\begin{align*}
l_3 &\coloneq \left(l_1 + \dfrac{g}{2}(l_2 - l_1)\right) + \underbrace{(g - 1)l_1 - (g + 1)l_2}_{\text{relation}} \\
&= \dfrac{g}{2}l_1 - \dfrac{g + 2}{2}l_2
\end{align*}
can also be realized as a generator of the Picard group. Once again we can check
that $-l_3 = \psi_\infty$ where $\psi_\infty$ is the $\psi$-class of the section
  which corresponds to the section  $s_\infty$ in the presentation of
  $\H_{g,2}^w$ as a quotient of the closed subvariety
  $$\HH_{sm}^w(2g+2,2) \subset  \HH_{sm}(2g+2,2) = \{(f,s_0,s_\infty)|f(0,1) =s_0^2, f(1,0) = s^2_\infty\}$$ described in Section~\ref{sec.2weierstrass}.
Moreover, by symmetry, it is easy to check that
$l_3' = -\frac{g + 2}{2}l_1 + \frac{g}{2}l_2$ also generates the Picard group.
Now we have $-l_3' = \psi_0$ where $\psi_0$ corresponds to the section $s_0$.

\subsection{Integral Chow ring of $\H_{g,n}^w$ for odd genus $g$ and $n = 1,2$}\hspace*{\fill}

\noindent We will then finish the computation of the integral Chow ring of $\H_{g,n}^w$ for odd genus $g$ and $n = 1,2$.

\subsubsection{$1$-Weierstrass-pointed case}
Recall that by Proposition \ref{action-1pt} if the genus $g$ is odd,
then $A^*(\H_{g,w})=A^*_T(U_w)$ where $U_w = (\Aff(2g + 1)\setminus \Delta') \setminus L$ with the action of $T = \G_m^2$ on $\Aff(2g + 1)$ defined by
\[(\alpha,\rho)\cdot f(x_0,x_1) = \alpha^{-2}\rho^{g}f(x_0/\rho,x_1).\]

Then by Proposition \ref{prop.disc} with
 $N = 2g + 1$, $a = -2,b = g$ and $(\alpha_0,\alpha_1) = (0,1), (\beta_0,\beta_1) = (0,0)$ we have 
\begin{align*}
A_T^*(\Aff(2g + 1) \setminus \Delta') &= \dfrac{\Z[l_1,l_2][\xi]}{\left(\alpha_{1,0}(\xi),\alpha_{1,1}(\xi),p(\xi),-2l_1 + gl_2 - \xi\right)}
\end{align*}
where
\[p(\xi) = \prod_{i = 0}^{2g + 1} [\xi - (2g - i + 1)l_2].\]
By substituting $\xi$ by $-2l_1 + gl_2$, we can write the generators of ideal as
\begin{align*}
& \alpha_{1,0}(-2l_1 + gl_2) = -[2g(l_1 + l_2) + 6gl_1], \\
& \alpha_{1,1}(-2l_1 + gl_2) = (-2l_1 + gl_2)^2 - l_2(-2l_1 + gl_2).
\end{align*}
The hyperplane $L$ is defined by the vanishing of the coefficient of $x_1^{2g+1}$and with the given weights its equivariant fundamental class is
$$[L]_T = -2l_1 + gl_2.$$
Therefore,
\[A^*_T(U_{w}) = \dfrac{\Z[l_1,l_2]}{(\alpha_{1,0}(-2l_1 + gl_2),\alpha_{1,1}(-2l_1 + gl_2),-2l_1 + gl_2,p(-2l_1 + gl_2))}.\]
But the following proposition can be easily checked.
\begin{prop}
The polynomial $p(-2l_1 + gl_2)$ and $\alpha_{1,1}$ are both in the ideal generated by $\alpha_{1,0}(-2l_1 + gl_2)$ and $-2l_1 + gl_2$.
\end{prop}
\begin{proof}
It is obvious that $\alpha_{1,1}$ is in the ideal generated by $-2l_1 + gl_2$ since $\alpha_{1,1} = \xi^2 - l_2\xi$. For $p(\xi)$, we can write
\begin{align*}
p(-2l_1 + gl_2) &= \prod_{i = 0}^{2g + 1}[-2l_1 + (-g + i - 1)l_2] \\
&= (-2l_1 + gl_2)\prod_{i = 0}^{2g}[-2l_1 + (-g + i - 1)l_2],
\end{align*}
and thus it is also in the ideal generated by $-2l_1 + gl_2$.
\end{proof}

Therefore for odd genus $g$,
\[A^*(\H_{g,w}) = \dfrac{\Z[l_1,l_2]}{(8gl_1 + 2gl_2,-2l_1 + gl_2)}.\]
Likewise, we have the natural ring homomorphism
\[\begin{tikzcd}[row sep = huge]
A^*(\H_{g}) = \dfrac{\Z[\tau,c_2,c_3]}{(4(2g + 1)\tau,8\tau^2 - 2(g^2 - 1)c_2,2c_3)} \ar[d,"\substack{\tau \mapsto l_1 \\ c_2 \mapsto -l_2^2 \\ c_3 \mapsto 0}"] \\
A^*(\H_{g,w}) = \dfrac{\Z[l_1,l_2]}{(8gl_1 + 2gl_2,-2l_1 + gl_2)}.
\end{tikzcd}\]
It can be checked that the image of generators $f_1(\tau) = 4(2g + 1)\tau$, $f_2(\tau,c_2) = 8\tau^2 - 2(g^2 - 1)c_2$ under the ring homomorphism are both vanishing in $A^*(\H_{g,w})$ since
\begin{align*}
f_1(l_1) &= 4(2g + 1)l_1 \\
&= 1 \cdot g_1 + (-2)\cdot g_2 
\end{align*}
and
\begin{align*}
f_2(l_1,-l_2^2) &= 8l_1^2 + 2(g^2 - 1)l_2^2 \\
&= l_2 \cdot g_1 + (-4l_1 + 2gl_2)\cdot g_2
\end{align*}
where $g_1 \coloneq 8gl_1 + 2gl_2$ and $g_2 \coloneq -2l_1 + gl_2$ are the generators of the ideal for $A^*(\H_{g,w})$. 

Moreover it can be checked similarly that the ring map is injective on Picard groups. 

In particular, by computing the Smith normal form of the degree one relations
and changing one of the generators to 
\[l_3' \coloneq l_1 - \left(\dfrac{g - 1}{2}\right)l_2,\]
we can express the integral Picard group of $\H_{g,w}$ as follows
\begin{align*}
\Pic(\H_{g,w}) &= \dfrac{\Z l_2 \oplus \Z l_3'}{(l_2 - 2l_3',4g(2g + 1)l_3')} \cong \dfrac{\Z l_3'}{(4g(2g + 1)l_3')}
\end{align*}
which means that the integral Picard group of $\H_{g,w}$ for odd genus $g$ is also cyclic of order $4g(2g + 1)$. \\

{\noindent\bf Natural generator for $A^*(\H_{g,w})$.}
There are multiple integral linear combinations of $l_1, l_2$
which generate the integral Picard group of $\H_{g,w}$. We now show that the
restriction to $\H_{g,w}$ of $\psi$-class on $\H_{g,1}$ generates $A^*(\H_{g,w})$.
\begin{cor}
  $A^*(\H_{g,w})$ is generated by $\psi$ and consequently
  \[A^*(\H_{g,w}) = \dfrac{\Z[\psi]}{4g(2g+1)\psi}.\]
\end{cor}
\begin{proof}
Set 
\begin{align*}
l_3 &\coloneq l_3' + \underbrace{\left[(-2)l_1 + gl_2\right]}_{\text{relation}} \\
&= (-1)l_1 + \left(\dfrac{g + 1}{2}\right)l_2,
\end{align*}

{\bf Claim.} The pullback of $\psi$ to $\H_{g,w}$ has Chow class
$l_1 - {g+1\over{2}}l_2$.

{\bf Proof of Claim.} As in the even genus case, the restriction of $\psi$ to $\H_{g,w}$
is the Chow class which associates to the Weierstrass-pointed hyperelliptic
curves $(C \to S, \sigma_w)$ the class $\sigma_w^*(c_1{\mathcal O}(-\sigma_w))$. On the
one hand, we may also consider the class $c_1({\mathcal O}_{\H_{g,1}}(\H_{g,w}))$. Its restriction to a family $(C \to S, \sigma_w)$ is the class
$\sigma_w^*(c_1({\mathcal O}(\sigma_w))) = -\psi$.
On the other hand we know that $\H_{g,w}$ is defined by the equation
$s=0$ in the quotient stack $[\HH_{sm}(2g+2,1)/(\G_m \times (B_2/\G_m))]$. This equation
has $T$-weight $\alpha^{-1} \rho^{{g+1\over{2}}}$. Thus,
$c_1({\mathcal O}_{\H_{g,1}}(\H_{g,w})) = -l_1 + {g+1\over{2}}l_2 \in A^*_T(\HH_{g,w}) = A^*(\H_{g,w})$.
The claim now follows.
\end{proof}

\subsubsection{$2$-Weierstrass-pointed case}
By Proposition \ref{action-2pt}, when $g$ is odd, we have $A^*(\H_{g,2}^w) = A^*_T(U_{w,2})$
where $U_{w,2} = (\Aff(2g) \setminus \Delta') \setminus (L_0 \cup L_\infty)$
and $T$ acts on $\Aff(2g)$ by
$$(\alpha,\rho) \cdot h(x_0,x_1) = \alpha^{-2} \rho^g h(x_0/\rho, x_1).$$

By Proposition \ref{prop.disc} with 
$N = 2g$, $a = -2, b = g$ and $(\alpha_0,\alpha_1) = (0,1),(\beta_0,\beta_1) = (0,0)$ we have that
\[A_T^*(\Aff(2g) \setminus \Delta') = \dfrac{\Z[l_1,l_2][\xi]}{(\alpha_{1,0}(\xi),\alpha_{1,1}(\xi),p(\xi),-2l_1 + gl_2 - \xi)}\]
where by substituting $\xi = -2l_1 + gl_2$, we have
\begin{align*}
& p(-2l_1 + gl_2) = \prod_{i = 0}^{2g}[-2l_1 - (g - i)l_2], \\
& \alpha_{1,0}(-2l_1 + gl_2) = -4(2g - 1)l_1, \\
& \alpha_{1,1}(-2l_1 + gl_2) = 4l_1^2 + g(g - 1)l_2^2 - 2(2g - 1)l_1l_2.
\end{align*}
The hyperplanes $L_0,L_\infty$ defined by setting the coefficients of $x_1^{2g},x_0^{2g}$ to zero respectively have the following equivariant fundamental classes. 
$$[L_0]_T= -2l_1 + gl_2,$$
$$[L_\infty]_T 
= -2l_1 - gl_2.$$
It can be checked that $p(-2l_1 + gl_2)$ and $\alpha_{1,0},\alpha_{1,1}$ are all in the ideal generated by the hyperplanes $[L_\infty]_T,[L_0]_T$. Thus
\begin{align*}
A^*(\H_{g,2}^w) &= \dfrac{\Z[l_1,l_2]}{(-2l_1 + gl_2,-2l_1 - gl_2)}.
\end{align*}
In particular, the integral Picard group of $\H_{g,2}^w$ can be computed as
\begin{align*}
\Pic(\H_{g,2}^w) = \dfrac{\Z l_3}{(4gl_3)},
\end{align*}
where $l_3$ can be taken either $l_3 = -l_1 - \frac{g + 1}{2}l_2$ or $l'_3 = -l_1 + \frac{g + 1}{2}l_2$, and both have order $4g$.
Once again it is easy to check that $-l_3$ and $-l'_3$ are the pullbacks of
the two $\psi$-classes from $\H_{g,2}$.

\subsection{Integral Chow ring of $[\M_{0,2g + 2}/S_{2g + 2 - n}]$ for $n = 1,2$}\hspace*{\fill}

\noindent We will now complete the proofs of Theorem \ref{thm.Hgw} and \ref{thm.Hg2w}.

\begin{thm}
The following results hold
\begin{itemize}
\item If $n = 1$, then
\[A^*([\M_{0,2g + 2}/S_{2g + 1}]) = \dfrac{\Z[l]}{(2g(2g + 1)l)}\]
where we can take the generator $l$ to be, for example $l = -l_1 + (g + 1)l_2$.

\item If $n = 2$, then
\[A^*([\M_{0,2g + 2}/S_{2g}]) = \dfrac{\Z[l]}{(2g l)}\]
where we can take the generator $l$ to be, for example either $l = -l_1 + (g + 1)l_2$ or $l = -l_1 - (g + 1)l_2$.
\end{itemize}
\end{thm}

\begin{proof}
The proof is essentially the same as the computation of the integral Chow ring of $\H_{g,n}^w$ for odd genus $g$ when $n = 1,2$ given the $T$-actions on $U_{w}$ and $U_{w,2}$.
\end{proof}

\section{Proof of Theorem \ref{thm.Hg3w}}

Let's now consider the quotient stack $[\M_{0,2g + 2}/S_{2g + 2 - n}]$ for $n \geq 3$. 

\begin{prop}
  If $n \geq 3$, the stack $[\M_{0,2g+2}/S_{2g+2-n}]$ is represented by an open subset of an affine space. In particular $A^*([\M_{0,2g+2}/S_{2g+2-n}]) = \Z$ if
  $n \geq 3$.
\end{prop}
\begin{proof}
 We recall that $\M_{0,2g + 2}$ is
 represented by the the scheme $(\Pro^1 \setminus \{0,1,\infty\})^{2g-1} \setminus \Delta$ where $\Delta$ is the big diagonal. Under this identification
  a point $(\Pro^1,s_1, \ldots , s_{2g+2}) \in \M_{0,2g+2}$ maps
  to $(p_4,\ldots , p_{2g+2})$ where $(0,1,\infty, p_4, \ldots, p_{2g+2})$
  is the image of the $(2g+2)$-tuple $(s_1, \ldots , s_{2g+2})$ under the unique
  automorphism of $\Pro^1$ that takes $(s_1,s_2,s_3)$ to $(0,1,\infty)$. 
  The action of $S_{2g+2}$ on $\M_{0,2g+2}$ which permutes the sections
acts on $(2g-1)$-tuple $(p_4, \ldots , p_{2g+1})$ as follows.
First let $\sigma \in S_{2g + 2}$ act on $(0,1,\infty,p_4,\ldots,p_{2g + 2})$ by permutation and then by apply the unique element of $\PGL_2$ which sends
$(\sigma(0), \sigma(1), \sigma(\infty))$ to $(0,1,\infty)$ to obain
a tuple $(0,1, \infty, q_4, \ldots, q_4)$. Then $\sigma(p_4, \ldots, p_{2g+2})
= (q_4, \ldots, q_{2g+2})$.

For any $n$, we identify $S_{2g+2-n}$ to be subgroup of $S_{2g+2}$ which fixes the first $n$ points. It follows from our description of the action that if $n \geq 3$ then $S_{2g+2-n}$ acts by permuting the tuple $(p_{n + 1}, \ldots , p_{2g+2}) \in (\Pro^1 \setminus \{0,1, \infty\})^{2g + 2 - n}$. Since the points are distinct, the action of
$S_{2g+2-n}$ is free and the quotient is the variety $((\Pro^1 \setminus \{0,1,\infty\})^{n - 3} \times \Sym^{2g + 2 - n}(\Pro^1 \setminus \{0,1,\infty\}))\setminus \Delta$ which is an open subvariety of $\Aff^{2g-1}$.
\end{proof}

\subsection{$3$-Weierstrass-pointed case}
For both $g$ even and odd, Proposition \ref{action-3pt} states that  $A^*(\H_{g,3}^w) = A^*_{\G_m}(U_{w,3})$
where $\G_m$ acts on $\Aff(2g-1)$ by
\[t \cdot h(x_0,x_1) = t^{-2}h(x_0,x_1).\]
With this action each of the hyperplanes
$L_0, L_1, L_\infty$ has the same equivariant fundamental class which is
$-2l$ where $l = c_1(\lambda)$.
It follows easily that 
\[A^*(\H_{g,3}^w) = \dfrac{\Z[l]}{(2l)}.\]
Moreover, the class $(-l)$ is the pullback any of the $\psi$-classes
from $\H_{g,3}$ and so can be viewed as the generator of the integral Chow ring of $\H_{g,3}^w$.

\subsection{The $n$-Weierstrass pointed case for $n > 3$.}
It remains to find a presentation for $A^*(\H_{g,n}^w)$ for $n >3$. In fact, we will compute $A^*(\H_{g,n}^w)$ for $n \geq 3$, but we need our earlier
presentation of $A^*(\H_{g,3}^w)$ to obtain a natural generator when $n > 3$.

Recall that the map $[\M_{0,2g+2}/S_{2g+1}] \to \D_{2g+2}$ realizes the map $[\M_{0,2g+2}/S_{2g+1}] \to [\M_{0,2g+2}/S_{2g+2}]$ as the universal effective Cartier divisor ${\mathcal B} \to \D_{2g+2}$ which is finite and \'etale of degree $2g+2$. Let $\M$ be the pullback of the line bundle ${\mathcal O}(-{\mathcal B})$ to $[\M_{0,2g+2}/S_{2g+1}]$.

\begin{prop}{(Dan Petersen)}\label{prop.danpetersen}
The universal Weierstrass divisor $\H_{g,w}$ is the $\mu_2$-gerbe obtained by taking the square root of $\M$. For $n \geq 1$
the gerbes $\H_{g,n}^w \to [\M_{0,2g+2}/S_{2g+2-n}]$ are the gerbes associated to the square root of the pullback of the line bundle $\M$ to $[\M_{0,2g+2}/S_{2g+2-n}]$.
\end{prop}
\begin{proof}
  The data $(P \to S, D_{2g+2})$ where $P \to S$ is a twisted $\Pro^1$ and
  $D_{2g+2}$ is a Cartier divisor which is finite and \'etale of degree $2g + 2$
  over $S$ corresponds to a morphism $S \to [\M_{0,2g+2}/S_{2g+2}]$.
  Let ${\mathcal P}$ be the $\mu_2$-gerbe parametrizing the square
  root of ${\mathcal O}(-D_{2g+2})$. Precisely, the line bundle ${\mathcal O}(-D_{2g+2})$ 
  corresponds to the morphism $P \to B\G_m$ and ${\mathcal P}$ is the gerbe obtained by base change along the square root gerbe
  $B \G_m \to B \G_m$ induced by map
  $\G_m \to \G_m$, $\lambda \mapsto \lambda^2$.
  By \cite{ArVi:03} we know that to give a $\mu_2$-cover of $P$ branched along
  $D_{2g+2}$ (\ie a hyperelliptic curve) is equivalent to giving a section $s \colon P \to {\mathcal P}$.

  The map $S \to [\M_{0,2g+2}/S_{2g+2}]$ lifts to a morphism $[\M_{0,2g+2}/S_{2g+1}]$
  when the \'etale covering $D_{2g+2} \to S$ admits a section $\sigma \colon S \to D_{2g+2}$. 
  Putting this together we see that the data of a section of $D_{2g+2} \to S$ together with a section $D_{2g+2} \to {\mathcal P}$ is equivalent to the data
  of double cover of $C \to P$ branched along $D_{2g+2}$ together with a Weierstrass section; \ie a section of $\H_{g,w}(S)$.

  The second statement follows from base change.
\end{proof}

\begin{prop}\label{prop.gerbe}
The $\mu_2$-gerbe $\H_{g,n}^w \to [\M_{0,2g + 2}/S_{2g + 2 - n}]$ is trivial if $n \geq 3$.
\end{prop}
\begin{proof}
The $\mu_2$-gerbe $\H_{g,n}^w \to [\M_{0,2g + 2}/S_{2g + 2 - n}]$ is the gerbe associated to the square root of a line bundle on $[\M_{0,2g + 2}/S_{2g + 2 - n}]$ by Proposition \ref{prop.danpetersen}. Since $[\M_{0,2g + 2}/S_{2g + 2 - n}]$ is represented by an open subvariety of an affine space, and line bundles over open sets of an affine space are trivial, it proves the result.
\end{proof}

It follows that  if $n\geq 3$, $\H_{g,n}^w = [\M_{0,2g + 2}/S_{2g + 2 - n}] \times B\mu_2$ by Proposition \ref{prop.gerbe}. Hence $A^*(\H_{g,n}^w) = A^*(B\mu_2) = \Z[l]/(2l)$ for $n \geq 3$. \\

{\noindent\bf Natural generator for $A^*(\H_{g,n}^w)$ when $n > 3$.} 
As in the case when $n = 1,2,3$, we will show that any restriction of $\psi$-class on $\H_{g,n}$ to $\H_{g,n}^w$ generates the integral Chow ring $A^*(\H_{g,n}^w)$.

\begin{cor}
  For $n \geq 3$, $A^*(\H_{g,n}^w)$ is generated by $\psi$ where
  $\psi$ is the pullback of any $\psi$-class from $\H_{g,n}$. Consequently $A^*(\H_{g,n}^w) = \Z[\psi]/(2\psi)$.
\end{cor}
\begin{proof}
  The proof proceeds by induction and it suffices to
  prove that $\psi \neq 0$ in $\Pic(\H_{g,n}^w)$. We have already proved the statement
  for $n =3$. Assume that the statement has been proved for $n = k+3$ with
  $k > 0$.
By Proposition \ref{prop.gerbe}, we have the following cartesian diagram
\[\begin{tikzcd}
\H_{g,k+4}^w \ar[r,"g_2"] \ar[d,"f"] & \left[(\M_{0,2g + 2}/S_{2g - 2-k})/\mu_2\right] \ar[d] \\
\H_{g,k+3}^w \ar[r,"g_1"] & \left[(\M_{0,2g + 2}/S_{2g - 1-k})/\mu_2\right]
\end{tikzcd}\]
where the map $f$ corresponds to forgetting any of the marked Weierstrass points
and the horizontal maps $g_1,g_2$ are both isomorphisms and $\mu_2$ acts trivially on both schemes $\M_{0,2g + 2}/S_{2g - 1-k}$ and $\M_{0,2g + 2}/S_{2g - 2-k}$. The pullback map on the Chow rings induced by the right vertical map is the identity map
so $f^*$ is an isomorphism and in particular it is not the zero map.
By induction we know that $\psi \neq 0$ in $A^*(\H_{g,k+3}^w)$. Hence
$f^*\psi$ is non-zero in $A^*(\H_{g,k+4}^w)$. Since the pullback of
a $\psi$-class under the map $f$ is a $\psi$-class the result follows.
\end{proof}



\begin{thebibliography}{1}

\bibitem{ArVi:03}
Alessandro Arsie and Angelo Vistoli.
\newblock Stacks of cyclic covers of projective spaces.
\newblock {\em Compos. Math.}, 140(3):647--666, 2004.

\bibitem{diL:18}
Andrea Di~Lorenzo.
\newblock The {C}how ring of the stack of hyperelliptic curves of odd genus.
\newblock {\em Int. Math. Res. Not. IMRN}, (4):2642--2681, 2021.

\bibitem{EdFu:07}
Dan Edidin and Damiano Fulghesu.
\newblock The integral {C}how ring of the stack of hyperelliptic curves of even
  genus.
\newblock {\em Math. Res. Lett.}, 16(1):27--40, 2009.

\bibitem{EdGr:98}
Dan Edidin and William Graham.
\newblock Equivariant intersection theory.
\newblock {\em Invent. Math.}, 131(3):595--634, 1998.

\bibitem{FuVi:11}
Damiano Fulghesu and Filippo Viviani.
\newblock The {C}how ring of the stack of cyclic covers of the projective line.
\newblock {\em Ann. Inst. Fourier (Grenoble)}, 61(6):2249--2275 (2012), 2011.

\bibitem{GoVi:06}
Sergey Gorchinskiy and Filippo Viviani.
\newblock Picard group of moduli of hyperelliptic curves.
\newblock {\em Math. Z.}, 258(2):319--331, 2008.

\bibitem{KlLo:79}
Knud L{\o}nsted and Steven~L. Kleiman.
\newblock Basics on families of hyperelliptic curves.
\newblock {\em Compositio Math.}, 38(1):83--111, 1979.

\bibitem{Lar:21}
Eric Larson.
\newblock{The integral Chow ring of  $\Mbar_2$}.
\newblock{\em Algebraic Geometry}, 8(3):286--318, 2021.

\bibitem{Per:21}
Michele Pernice.
\newblock {The Integral Chow Ring of the Stack of 1-Pointed Hyperelliptic
  Curves}.
\newblock {\em International Mathematics Research Notices}, (15):11539--11574, 2022..

\end{thebibliography}
\end{document}